\documentclass[reqno, 11pt]{amsart}
\usepackage{ifthen}

\usepackage{setspace}
\usepackage[a4paper, left=3cm, right=3cm, top=4cm]{geometry}

\usepackage{amssymb, amsthm, epsf, epsfig}
\usepackage{amsmath}
\usepackage{verbatim}
\usepackage{caption}
\usepackage{subcaption}
\usepackage{dsfont}
\usepackage{wrapfig}
\usepackage[shortlabels]{enumitem}
\usepackage[usenames,dvipsnames]{xcolor}
\usepackage{esvect}
\usepackage{xfrac}
\usepackage{wrapfig}

\newtheorem{theorem}{Theorem}
\newtheorem{proposition}{Proposition}

\theoremstyle{definition}
\newtheorem{definition}{Definition}
\newtheorem{example}{Example}

\theoremstyle{remark}

\newcommand{\map}[1]{\textbf{#1}}

\numberwithin{equation}{section}

\title{Singularly perturbed discrete differential equations}

\author{Michael Drmota}
\author{Eva-Maria Hainzl}
\thanks{This work was partially founded by the Austrian Science
Foundation FWF, projects P 35016, F 100203 and ANR LOUCCOUM.}

\begin{document}

\maketitle 

\begin{abstract}
Discrete differential equations
appear most prominently in planar map and lattice path enumeration. In this work we consider
discrete differential equations with an additional parameter $x$, where
the order of the equation is $1$ for $x=0$ but $k> 1$ for $x\ne 0$. 
We call such equations singularly perturbed. 
The main contribution of this work is to show that there is actually a smooth transition
under certain natural assumptions. As an application of this result we consider
pattern counts in triangular planar maps and derive a central limit theorem
for these counts.
\end{abstract}

%++++++++++++++++++++++++++++++++++++++++++++++++++++++++++
%++++++++++++++++++++++++++++++++++++++++++++++++++++++++++

\section{Introduction} 

Discrete differential equations are a type of functional equation for a multivariate generating function $F(z,u)=\sum_{n\geq 0}F_n(z)u^n$ and are characterized by the appearance of \emph{discrete derivatives}\footnote{Some authors also refer to them as \emph{divided differences}.} of $F(z,u)$, defined as $\Delta^k F(z,u) = \sum_{n\geq k}F_n(z)u^{n-k}$.
    We call an equation 
    \[
        F(z,u) = zQ(z,u,F(z,u),\Delta F(z,u),\dots, \Delta^kF(z,u))
    \]
    for some $k \in \mathbb{N}$, \emph{discrete differential equation} (DDE). They are also known as catalytic equations. The order of the DDE is the highest order discrete derivative it contains.
    
    They appear most prominently in planar map enumeration~\cite{Tutte1963, Tutte1962c,Tutte1962a,BrownTutte} but also in statistical mechanics~\cite{Temp}, in the enumeration of lattice paths~\cite{BandFla2002, bm-walks, Bostanhab}, pattern avoiding permutations~\cite{Gire, Gui}, stack-sortable permutations~\cite{stacks2, stacks}, certain types of polyominoes~\cite{poly, bm-pol}, parking trees~\cite{park_tree_uni} and various other structures which are often in bijection with the former, such as Tamari intervals~\cite{Chap, Tamar} and fighting fish~\cite{Duchi}.

    In~\cite{DrmotaNoyYu,DHai}, the authors showed that solutions to first and second order DDEs show universal singular behavior and exploited this fact to derive central limit theorems for additional parameters $x$ that perturb the equation \emph{regularly}. That is, setting the parameter $x=1$ would not change the order of the DDE. Recently, Contat and Curien \cite{curcon} proved 
    a remarkable general (universal) asymptotic result for positive DDEs of arbitrary order but they did not consider additional parameters.

    In this work, we consider DDEs with an additional parameter $x$ \emph{singularly} perturbing a discrete differential equation such that
    \begin{equation*}
        F(z,u,x) = zQ\big(z,u,F(z,u,x),\Delta F(z,u,x)\big)+(x-1)zR\big(z,u,x,F(z,u,x),\dots, \Delta^kF(z,u,x)\big),
    \end{equation*}
    where $k>1$ and $Q$ and $R$ are rational functions. It is essential that the order of the DDE changes from $1$ for $x=1$ to $k > 1$ for $x\ne 1$. A solution theory of higher order DDE's is much more involved than for degree $1$ and it is a-priori not clear that there is a smooth transition from $x=1$ to $x\ne 1$ that is actually needed proving a central limit theorem for an associated random variable $X_n$ defined by
    \[
        \mathbb{P}\left(X_n = k\right) = \frac{[z^nx^k]F(z,0,x)}{[z^n]F(z,0,1)}, \quad k\geq 0.
    \]
    A natural application of such a central limit theorem appears in map enumeration, where $x$ marks, for example, a pattern occurrence of a fixed map $\map{p}$ (see Theorem~\ref{thm:patts}). The goal is to prove a central limit theorem for the random variable $X_n$ that counts the number of such pattern occurrences in random maps of size $n$, where every map of size $n$ is equally likely. 

    In the next section, we illustrate the singular perturbation along a simple example and indicate how the general situation can be handled. In Section~\ref{sec:trian}, we then deduce a central limit theorem for pattern counts of maps which cannot self-intersect in a random simple triangulation. Finally, we discuss future extensions of our results.
    
\section{Singularly perturbed discrete differential equations}

In what follows we will analyze a DDE for the unknown function $F(z,u,x)$ of the form
$F(z,u,x) = \sum_{i\geq 0} f_i(z,x)u^i$
    \begin{align}
        F(z,u,x) &= z Q\left(z,u, F(z,u,x), \Delta F(z,u,x) \right)    \label{eqpertDDE} \\
        &\qquad +(x-1)zR\big(z,u,x,F(z,u,x),\dots, \Delta^kF(z,u,x)\big), \nonumber
    \end{align}
    where $\Delta^k F(z,u,x) = \sum_{i\geq k} f_k(z,x)u^{i-k}$. 
We also assume that the functions $Q(z,u,y_0,y_1)$ is rational with non-negative coefficients and
that $R(z,u,y_0,y_1,y_2)$ rational. Both functions should be analytic in a sufficiently large 
region that contains the singularities and singular values of the solution function.
Furthermore we suppose that $Q(z,0,0,0) \ne 0$,  $Q_u(z,u,y_0,0) \ne 0$ and $Q_{y_0y_0}(z,u,y_0,y_1) \ne 0$ or $Q_{uy_0}(z,u,y_0,y_1) \ne 0$. 
We also assume that the coefficients of the power series expansion of 
$F(z,u,x)$ are non-negative.\footnote{By assumption on $Q$ this non-negativity property is immediate for the function $F(z,u,1)$, that is, for $x=1$. 
The non-negativity condition for $F(z,u,x)$ is not that easy to check in general.
However, if the coefficients of $F(z,u,x)$ have a combinatorial interpretation -- as it 
will be in our applications -- this condition is automatically satisfied.} 
Finally, we assume that the dependency graph of the equation (\ref{eqpertDDE}) is strongly connected for $x=1$. For this purpose we expand the left-hand and right-hand side of
(\ref{eqpertDDE}) into powers of $u$ and obtain an infinite system of equations for
$f_i(z,1)$:
\[
    f_i(z,1) = \tilde Q_{i}(z,1, f_0(z,1), \ldots, f_{i+k}(z,1)), \qquad (i\ge 0).
\]
We then associate to this infinite system a so-called dependency graph on the 
vertex set $V = \{0,1,\ldots\}$ with directed edges from $i$ to $j$ is $\tilde Q_i$ depends
on $f_j$. 

We recall that the asymptotic behavior of $[z^n] F(z,0,1) = [z^n] f_0(z,1)$ can
be completely characterized. From \cite{DrmotaNoyYu} we know that we have
\[
[z^n] f_0(z,1) \sim c_j\, n^{-5/2} \rho^n  \qquad (n\to \infty, \ n\equiv j \bmod d)
\]
for positive constants $c_j, \rho$ and for some residue classes $j \bmod d$, 
$j\in J$, where $d \ge 1$. Actually in \cite{DrmotaNoyYu} only the polynomial case is handled,
however, if $Q(z,u,y_0,y_1)$ is rational with a sufficiently large analyticity range
the same methods apply and, thus, the same result holds.

\begin{theorem}\label{thm:clt1dde}
    Let $F(z,u,x) = \sum_{i\geq 0} f_i(z,x)u^i$ be implicitely given by
    the DDE (\ref{eqpertDDE}), where $k\ge 2$,  $Q(z,u,y_0,y_1)$ is rational with non-negative coefficients, $R(z,u,y_0,y_1,\ldots, y_k)$ rational but linear in $y_k$ and both with large enough region of convergence. Further, let $Q(z,0,0,0) \ne 0$,  $Q_u(z,u,y_0,0) \ne 0$, 
    $Q_{y_1}(z,u,y_0,y_1) \ne 0$, and $Q_{y_0y_0}(z,u,y_0,y_1) \ne 0$ or $Q_{uy_0}(z,u,y_0,y_1) \ne 0$ and let the dependency graph of the equation be strongly connected for $x=1$.
    If $X_n$ is defined by
    \[
        \mathbb {P}[X_n = k] = \frac{[z^n x^k]\, F(z,0,x)  }{[z^n]\, F(z,0,1)}
        \qquad (n \equiv j \bmod d, \ j\in J)
    \]
    then $X_n$ ($n \equiv j \bmod d$, $j\in J$) satisfies a central limit theorem of the form
    \[
    \frac{X_n - \mu n}{\sqrt n} \to N(0,\sigma^2)  \qquad (n \equiv j \bmod d, \ j\in J)
    \]
    where $\mu, \sigma$ are non-negative constants and
    $\mathbb{E} X_n \sim \mu n$ and $\mathbb{V} X_n \sim \sigma^2 n$
    ($n \equiv j \bmod d$ with $j \in J$).
\end{theorem}

We will first consider an example and illustrate the proof of  
Theorem~\ref{thm:clt1dde} along this case in detail.

\begin{example} Let $T(z,u,x) = \sum_{k\geq 0} t_k(z,x)u^k$ be implicitly defined by
\[
    T(z,u,x) = 1+z^2uT(z,u,x)^2+z\Delta T(z,u,x)+ zx \Delta^2 T(z,u,x).
\]
For the sake of notational simplicity we have changed $x-1$ to $x$. The {\it critical} value is now $x=0$ (and not $x=1$).
\end{example}
First we analyze quickly the (critial) case $x=0$. Here we have the equation
\[
T = 1+z^2uT^2+z\Delta T 
\]
or in polynomial terms
\[
P_1(z,u,T,t_0) = u(1+z^2uT^2+z\Delta T-T) = u+z^2u^2T^2+z(T-t_0)-uT
\]
(with $t_0(z) = T(z,0,0)$). By the theory of \cite{BMJ} we have to solve the system
\begin{equation}\label{eqP1system}
 P_1(z,u,\tilde T,t_0) = 0, \quad P_{1,u}(z,u,\tilde T,t_0) = 0, \quad P_{1,T}(z,u,\tilde T,t_0) = 0
\end{equation}
to obtain functions $u(z)$, $t_0(z)$, $\tilde T(z) = T(z,u(z),0)$. 
By \cite{DrmotaNoyYu} we know that these three functions have a common dominant singularity 
$z_0$. More precisely, the functions $u(z)$ and $\tilde T(z)$ have a square-root singularity 
at $z_0$ whereas $t_0(z) = T(z,0,0)$ has a $3/2$-singularity:
\[
t_0(z) = g_0(z) + h_0(z) (1 - z/z_0)^{3/2}
\]
with analytic functions $g_0(z), h_0(z)$ that satisfy $g_0(z_0) \ne 0, h_0(z_0) \ne 0$.
The singularity $z_0$ can be determined by system (\ref{eqP1system}) together with 
the equation
\begin{equation}
    \label{eqP1system-erg}
    \left( P_{1,uu}P_{1,TT} -P_{1,uT}^2  \right)(z,u,\tilde T, t_0) = 0,
\end{equation}
compare with \cite{DrmotaNoyYu}. Actually (\ref{eqP1system}) together with
(\ref{eqP1system-erg}) determine also the values $u(z_0)$, $\tilde T(z_0) = T(z_0,u(z_0))$, 
and $t_0(z_0) = T(z_0,0)$.

There might be further singularities of this type on the circle of convergence $|z| = z_0$, however, they will differ only by a $d$-th root of unity: $z_j = z_0 \exp(2\pi i j/d)$ with implies then
that the asymptotic expansion of the $n$-th coefficient is given by
\[
[z^n]t_0(z) \sim c_j\, n^{-5/2} z_0^{-n} \qquad (n\to \infty, \ n\equiv j \bmod d)
\]
with $j\in J$ positive constants $c_j$.

The main goal is to show that the function $T(z,0,x)$ behaves similarly if $x\approx 0$, that 
is, for complex $x$ close to $0$ we have
\begin{equation}\label{eq32expansion}
T(z,0,x) = g(z,x) + h(z,x) (1 - z/z_0(x))^{3/2}   
\end{equation}
for proper analytic functions $g(z,x)$, $h(z,x)$, $z_0(x)$ with
$g(z,0) = g_0(z)$, $h(z,0) = h_0(z)$, and $z_0(z) = z_0$. From this representation it 
immediately follows that
the coefficients satisfy 
\[
[z^n]T(z,0,x) \sim c_j(x) \, n^{-5/2} z_0(x)^{-n} \qquad (n\to \infty, n\equiv j \bmod d)
\]
with $j\in J$ and with analytic functions $c_j(x)$ with $c_j(0) = c_j$.
By standard theory (by applying Hwang's Quasi Power Theorem, for example, see \cite{DrmotaTrees}) 
the central limit theorem follows then directly (provided that $z_0'(1) \ne 0$). 

The essential point is that second order DDE's (and higher order DDE's, too) need
a more involved treatment than first order DDE's. We set now
\begin{align*}
P(z,u,T, t_0,t_1,x) &= u^2 \left( 1+z^2uT^2+z\Delta T+ zx \Delta^2 T -  T \right)  \\
&= u P_1(z,u, T,t_0) + zx ( T - t_0 - ut_1) \\
&= u^2+z^2u^3T^2+zu(T-T_0) +zx ( T - t_0 - ut_1)  -u^2 T,    
\end{align*}
where $t_0(z,x) = T(z,0,x)$ and $t_1(z,x) = T_u(z,0,x)$.\footnote{The factor $u^2$ is just used
to have a polynomial equation. This is not really necessary but convenient.}
Here we have to solve the system of equations (see \cite{BMJ})
\begin{align*}
P(z,u_1,\tilde T_1,t_0,t_1,x) &= 0, \quad P_u(z,u_1,\tilde T_1,t_0,t_1,x) &= 0, \quad 
P_T(z,u_1,\tilde T_1,t_0,t_1,x) &= 0, \\
P(z,u_2,\tilde T_2,t_0,t_1,x) &= 0, \quad P_u(z,u_2,\tilde T_2,t_0,t_1,x) &= 0, \quad 
P_T(z,u_2,\tilde T_2,t_0,t_1,x) &= 0
\end{align*}
for the unknown functions
$u_1(z,x)$, $u_2(z,x)$, $t_0(z,x)$, $t_1(z,x)$, $\tilde T_1(z,x) = T(z,u_1(z,x),x)$, and
$\tilde T_2(z,x) = T(z,u_2(z,x),x)$.

Since this is an algebraic system we can apply a proper elimination procedure
that leads to the following equation for $u_1(z,x)$ and $u_2(z,x)$ that
are determined by the property that $u_1(0,x) = u_2(0,x) = 0$ (and it will
be clear from the structure of the equations that there are precisely 
two different solutions of that kind). 
\begin{align*}
    0 = 2u^8z^3- u^6z^2 + u^5z^3&+xu^4z(54u^{6}z^2 - 36u^4z 
    + 42u^3z^2  + 4u^2 
    - 12uz  
    + 9z^2)\\
    &+x^2u^3( - 36u^4z^2 
     + 78u^3z^3  - 4u^3  + 20u^2z  
     - 42uz^2  
     + 30z^3)\\
    &+x^3zu^2(38u^3z^2+ 16u^2 
     - 52uz
     + 46z^2) \\
     &+x^4uz^2( - 21u 
      + 33z) + 9x^5z^3
\end{align*}
\begin{figure}[h]
    \centering
    \includegraphics[width=0.2\linewidth]{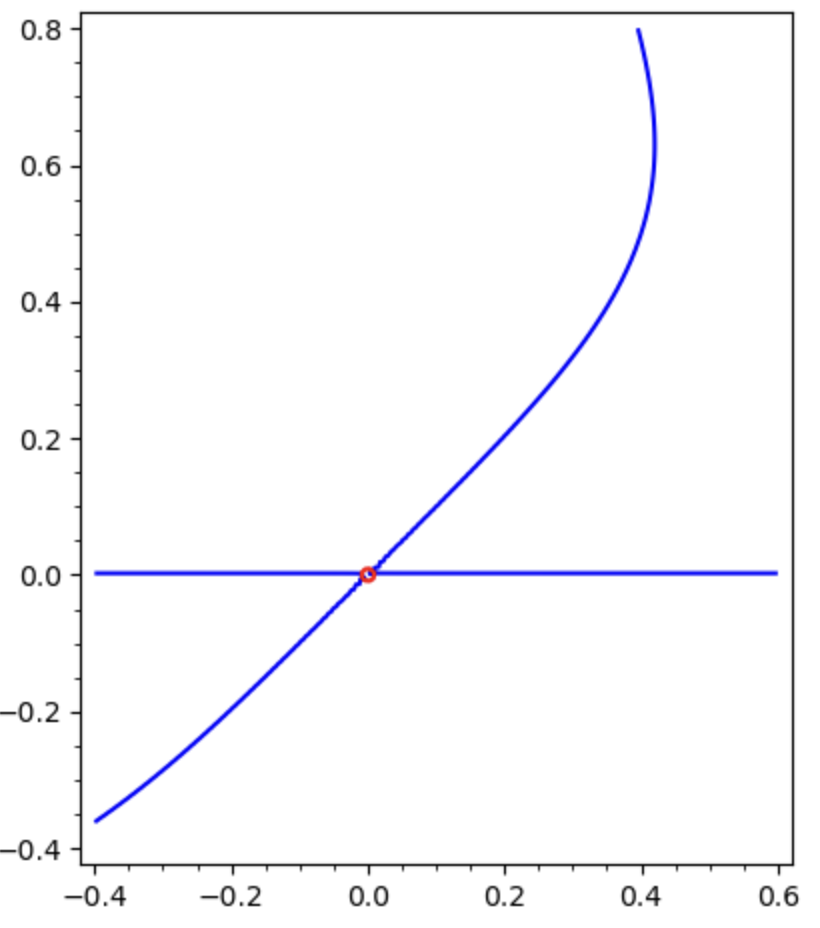}
    \includegraphics[width=0.2\linewidth]{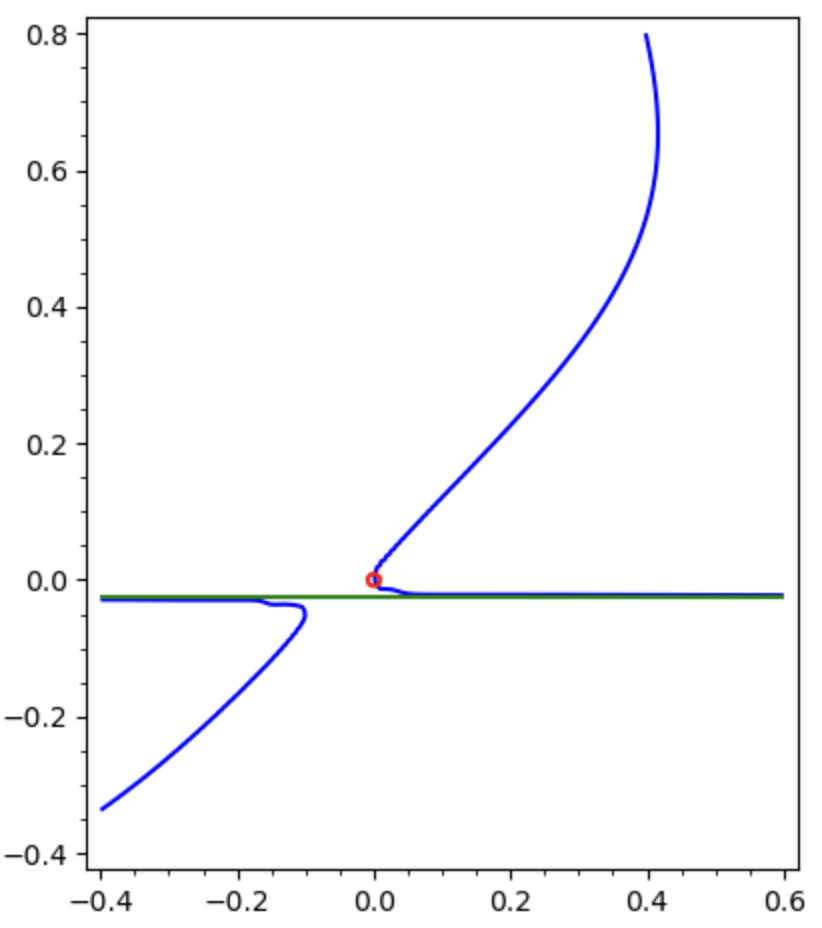}
    \includegraphics[width=0.2\linewidth]{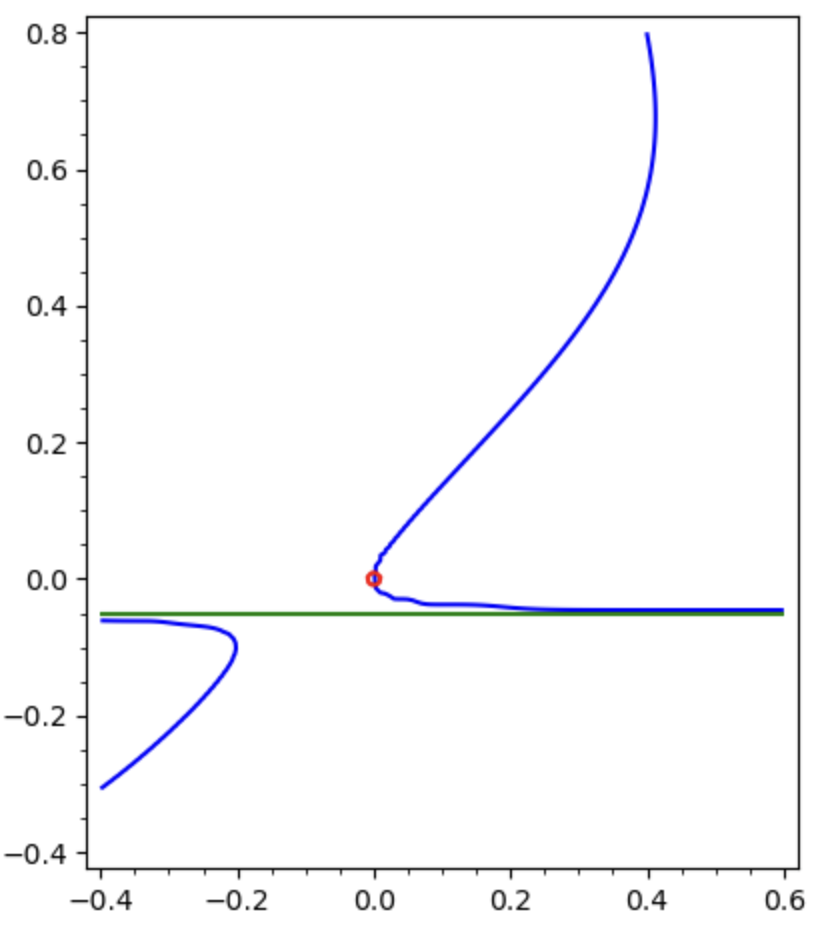}
    \includegraphics[width=0.2\linewidth]{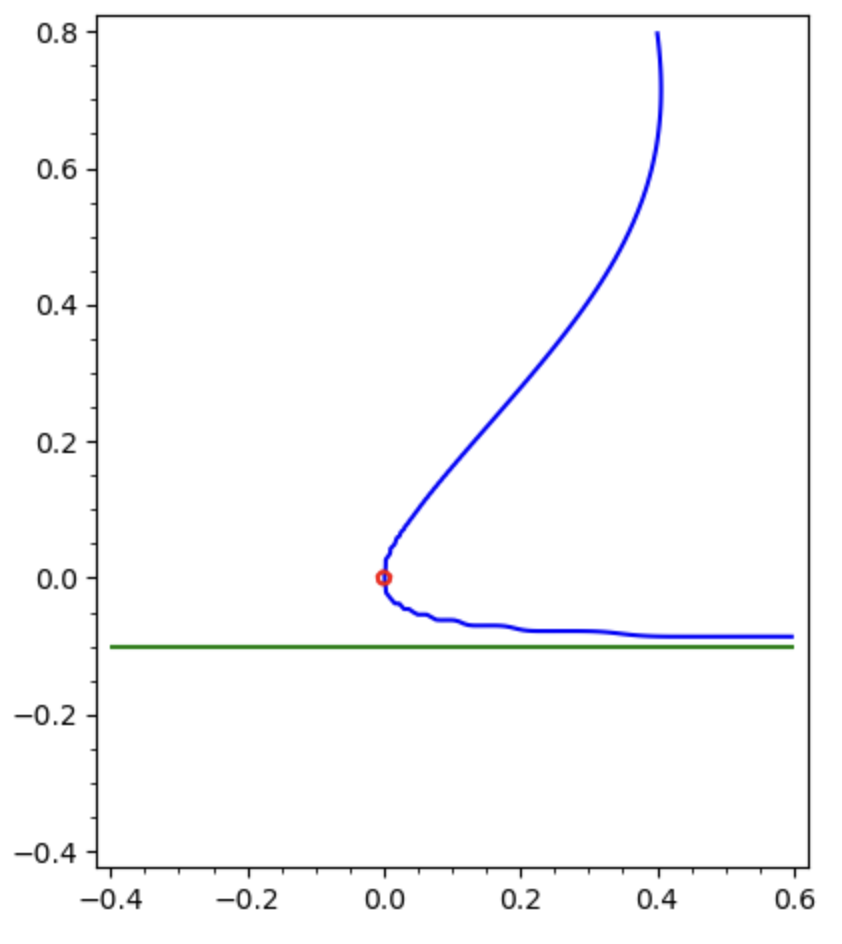}\\ \vspace{2mm}
    \includegraphics[width=0.2\linewidth]{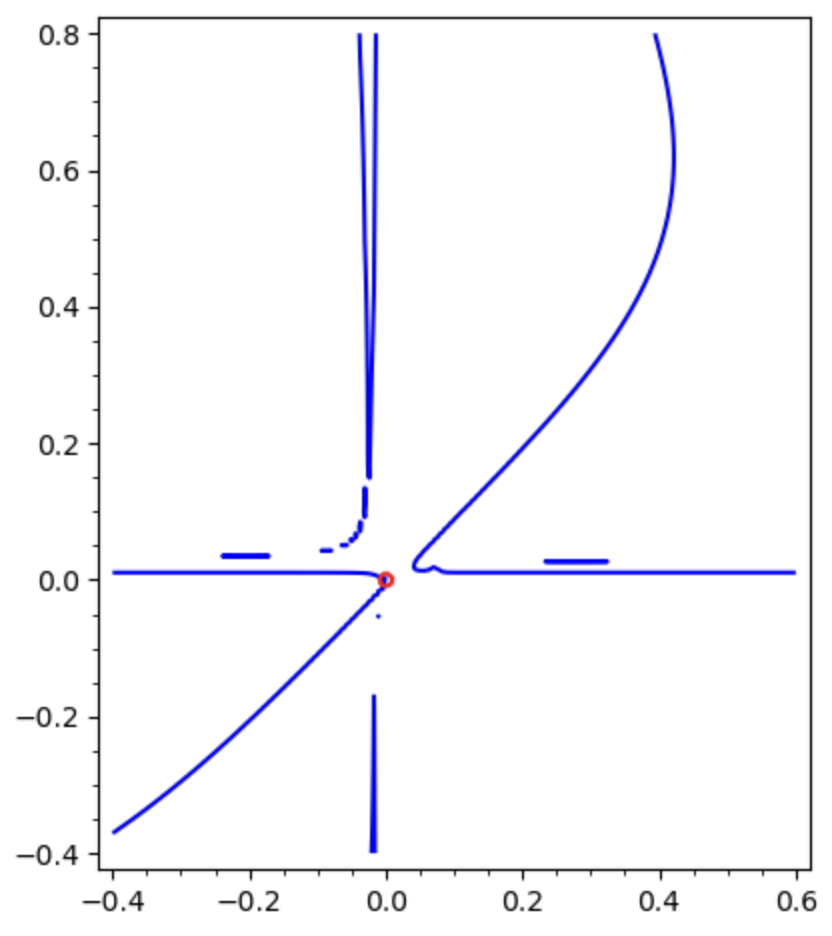}
    \includegraphics[width=0.2\linewidth]{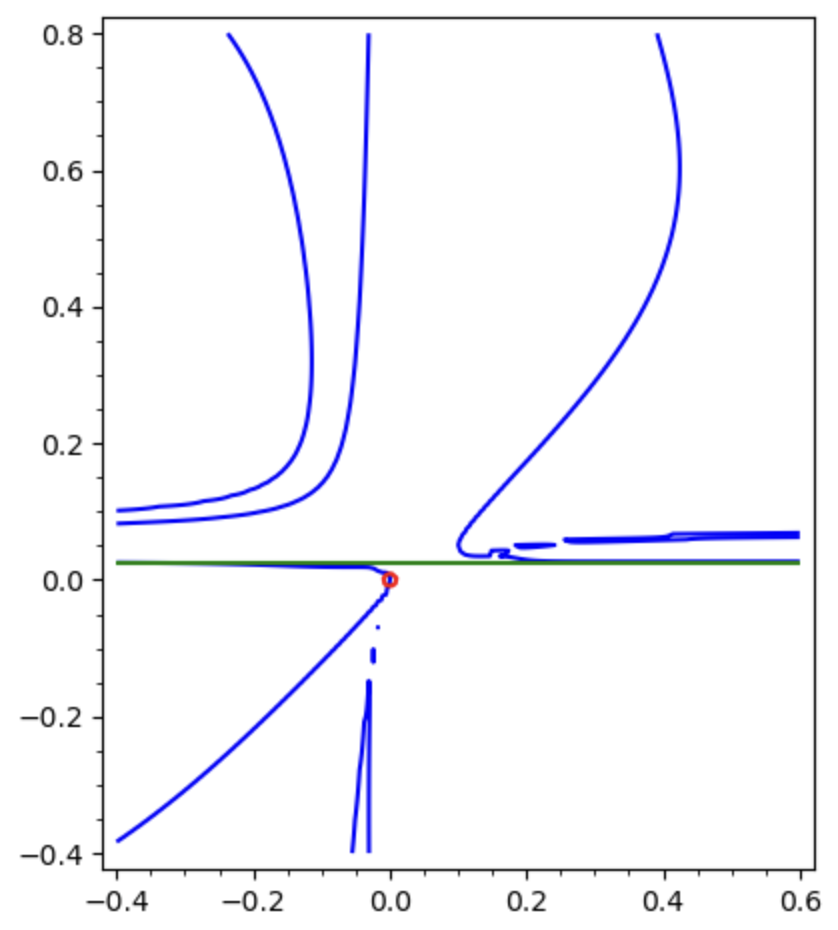}
    \includegraphics[width=0.2\linewidth]{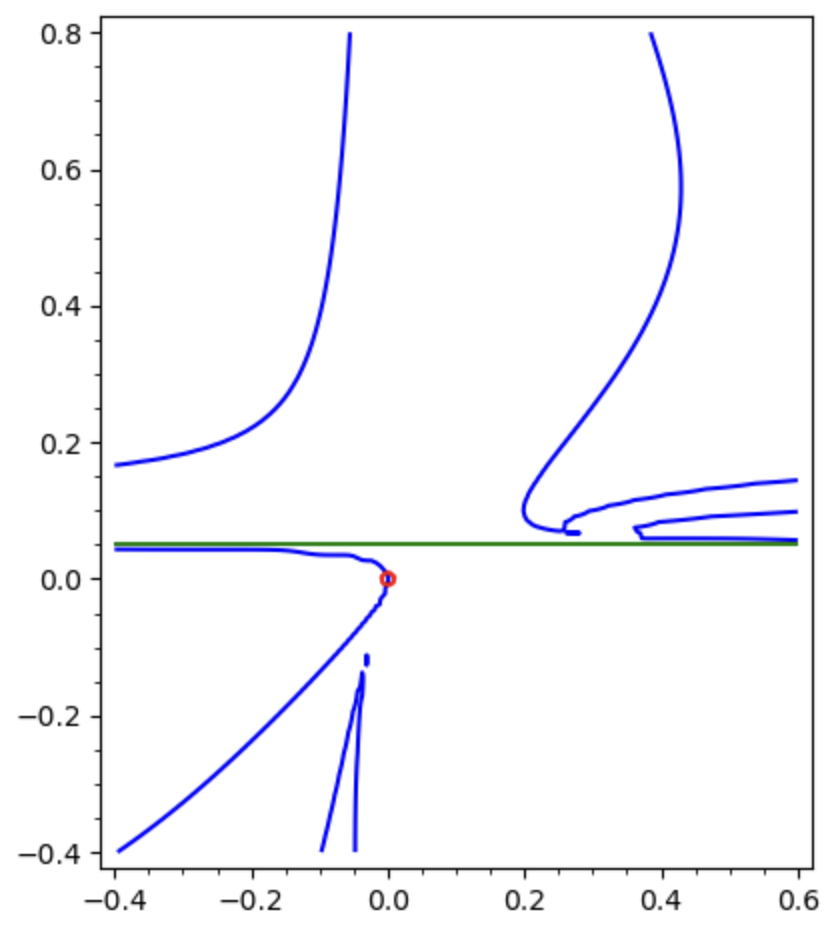}
    \includegraphics[width=0.2\linewidth]{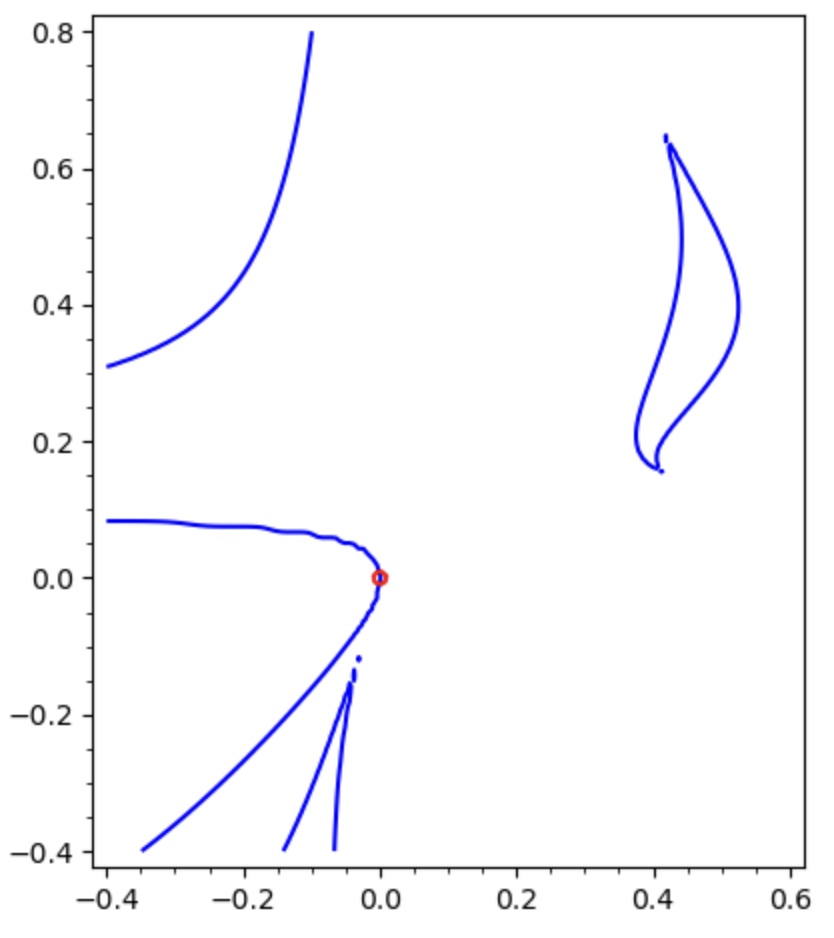}
    \caption{The solutions $u_1(x), u_2(x)$ for the values $x= 0, 0.025, 0.05, 0.1$ in the first line and $x= -0.01, -0.025, -0.05, -0.1$ in the second line.\\
    The red dot marks the origin, the green line the negative value of $x$ to show that $0>u_2(z) > -x$ for $z>0$.}
    \label{fig:plots_u0}
\end{figure}
In Figure~\ref{fig:plots_u0} we can see the (real valued) solutions to this equation for different values of $x$. As $x\rightarrow 0$ two separate curves meet at the origin and create the solutions $u_1(z,0) = u_1(z)$ and $u_2(z,0)=0$.  

The corresponding equation for $t_0(z,x) = T(z,0,x)$  is
\begin{align*}
    0 &= 11664t_0^5x^4z^{10} - 7776t_0^4x^3z^8 - 864t_0^4x^2z^9 + 16200t_0^3x^3z^8 - 1728t_0^3x^3z^5 + 1080t_0^3x^2z^6 \\
    &\quad + 288t_0^3xz^7 + 16t_0^3z^8 + 21600t_0^2x^3z^5 + 1620t_0^2x^2z^6 - 24t_0^2xz^7 - 45000t_0x^3z^5  \\
    &\quad - 1500t_0x^2z^6 + 25000x^3z^5 + 576t_0^2x^2z^3 + 136t_0^2xz^4 + 8t_0^2z^5 - 2640t_0x^2z^3 - 602t_0xz^4 \\
    &\quad - 36t_0z^5 + 2000x^2z^3 + 450xz^4 + 27z^5 + 64t_0x^2 + 16t_0xz + t_0z^2 - 64x^2 - 16xz - z^2
\end{align*}
The implicit plot around $t_0(0,x) = T(0,0,x)=1$ is shown in Figure~\ref{fig:plots_t0} for different values of $x$. As one can see, the trajectory of $t_0(z,x) = T(z,0,x)$ only differs slightly through the small perturbation by $x$.

\begin{figure}[h]
    \centering
    \includegraphics[width=0.2\linewidth]{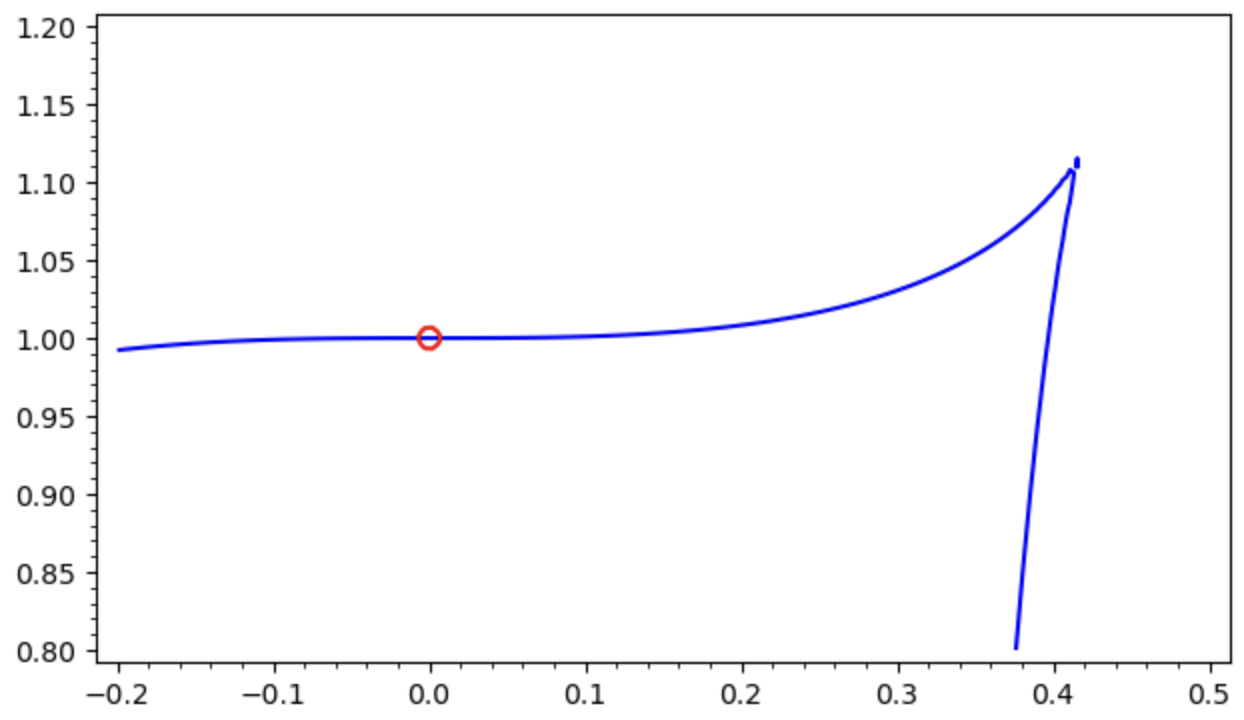}
    \includegraphics[width=0.2\linewidth]{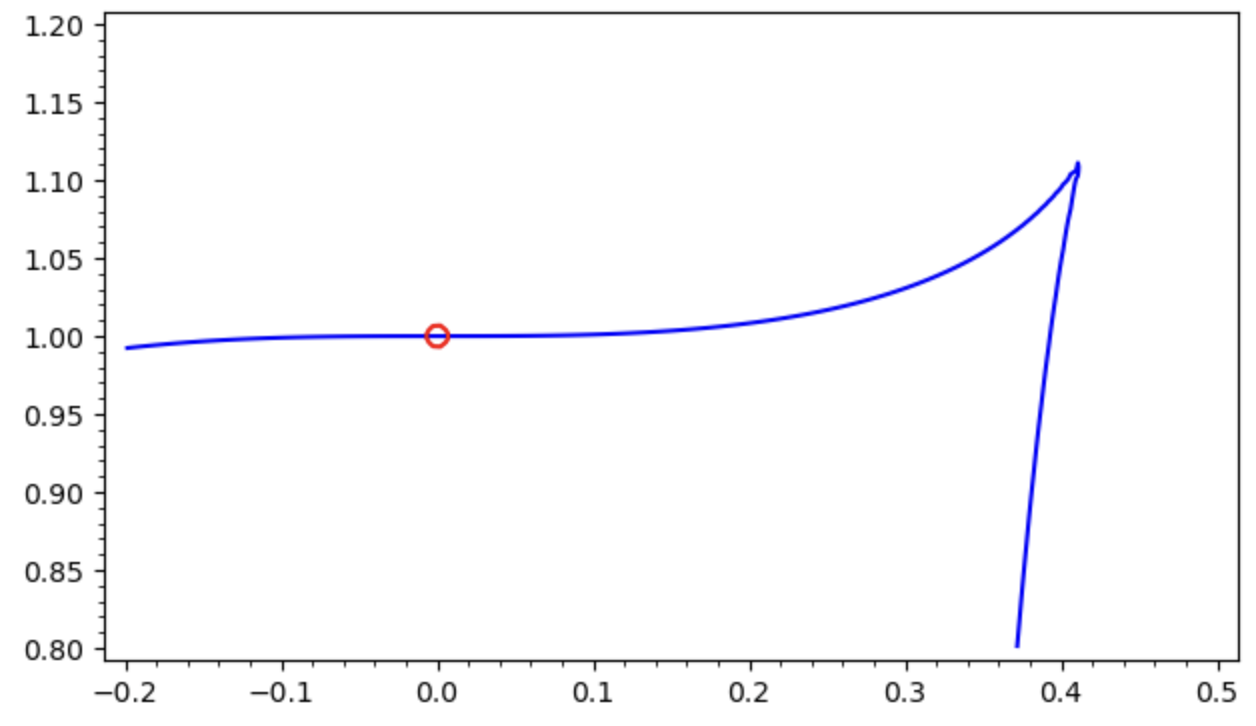}
    \includegraphics[width=0.2\linewidth]{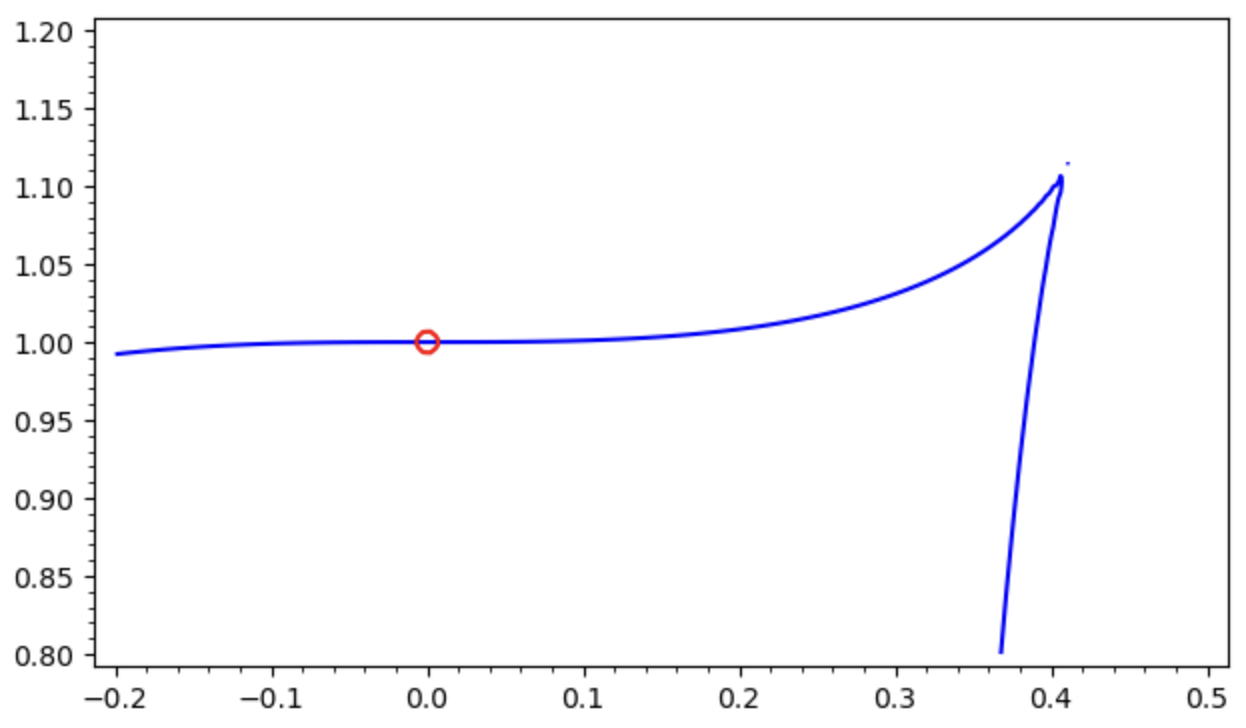}
    \includegraphics[width=0.2\linewidth]{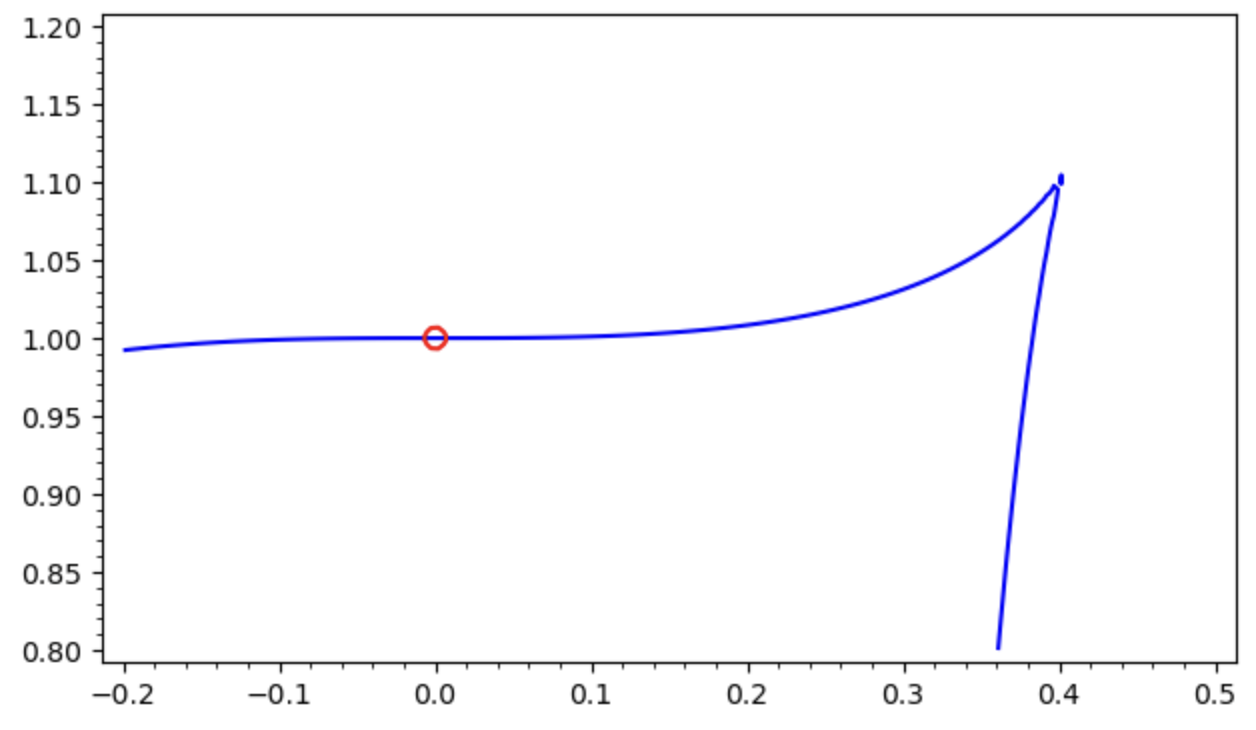}\\ \vspace{4mm}
    \includegraphics[width=0.2\linewidth]{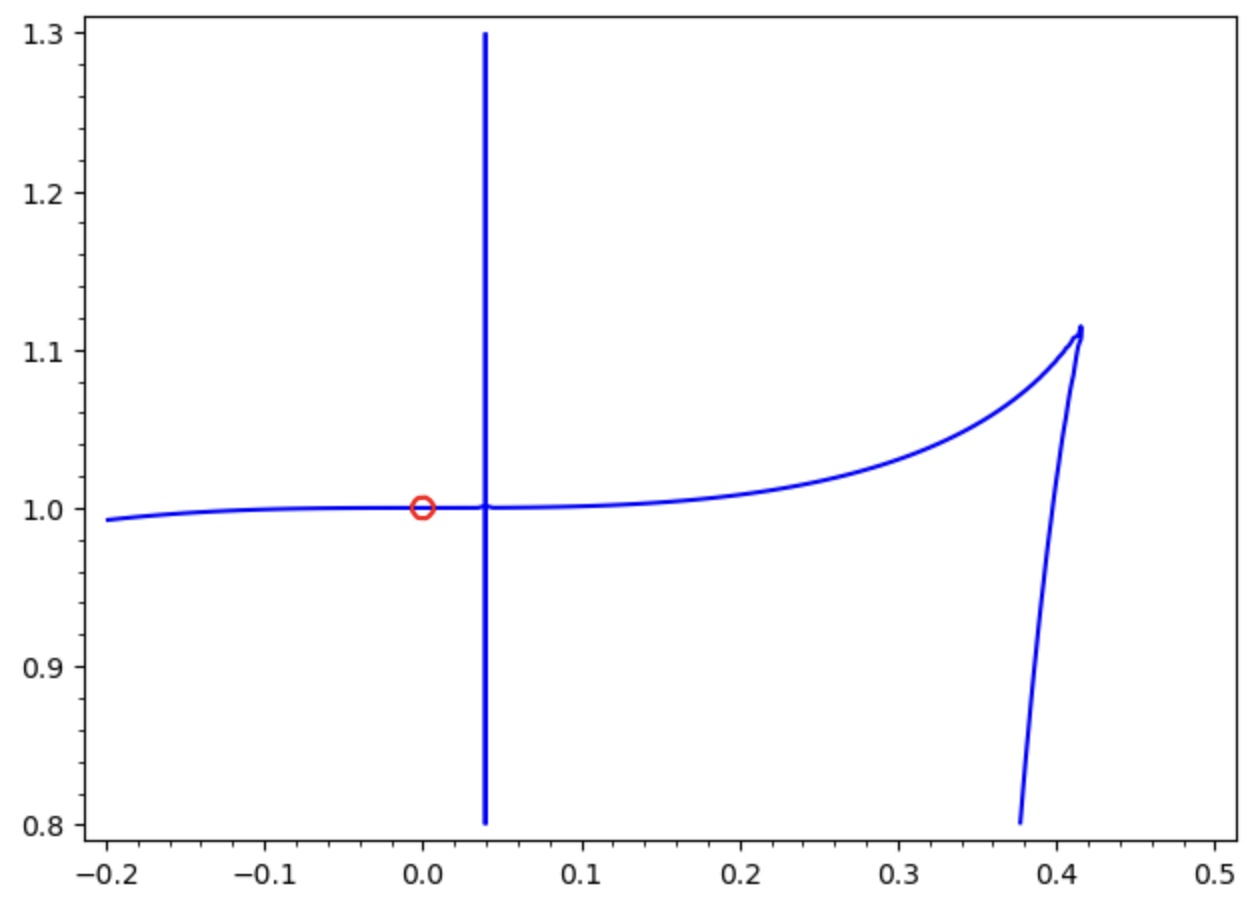}
    \includegraphics[width=0.2\linewidth]{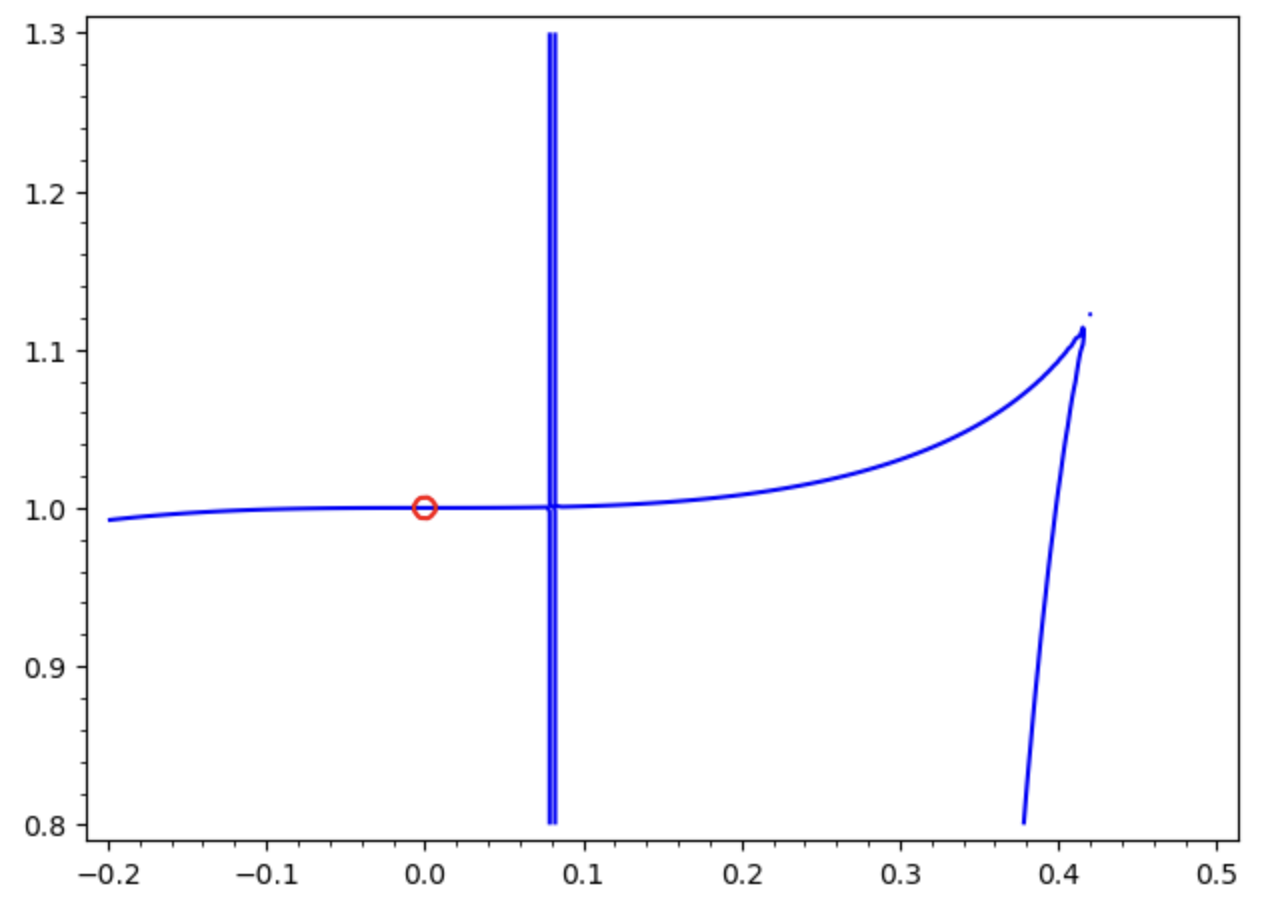}
    \includegraphics[width=0.2\linewidth]{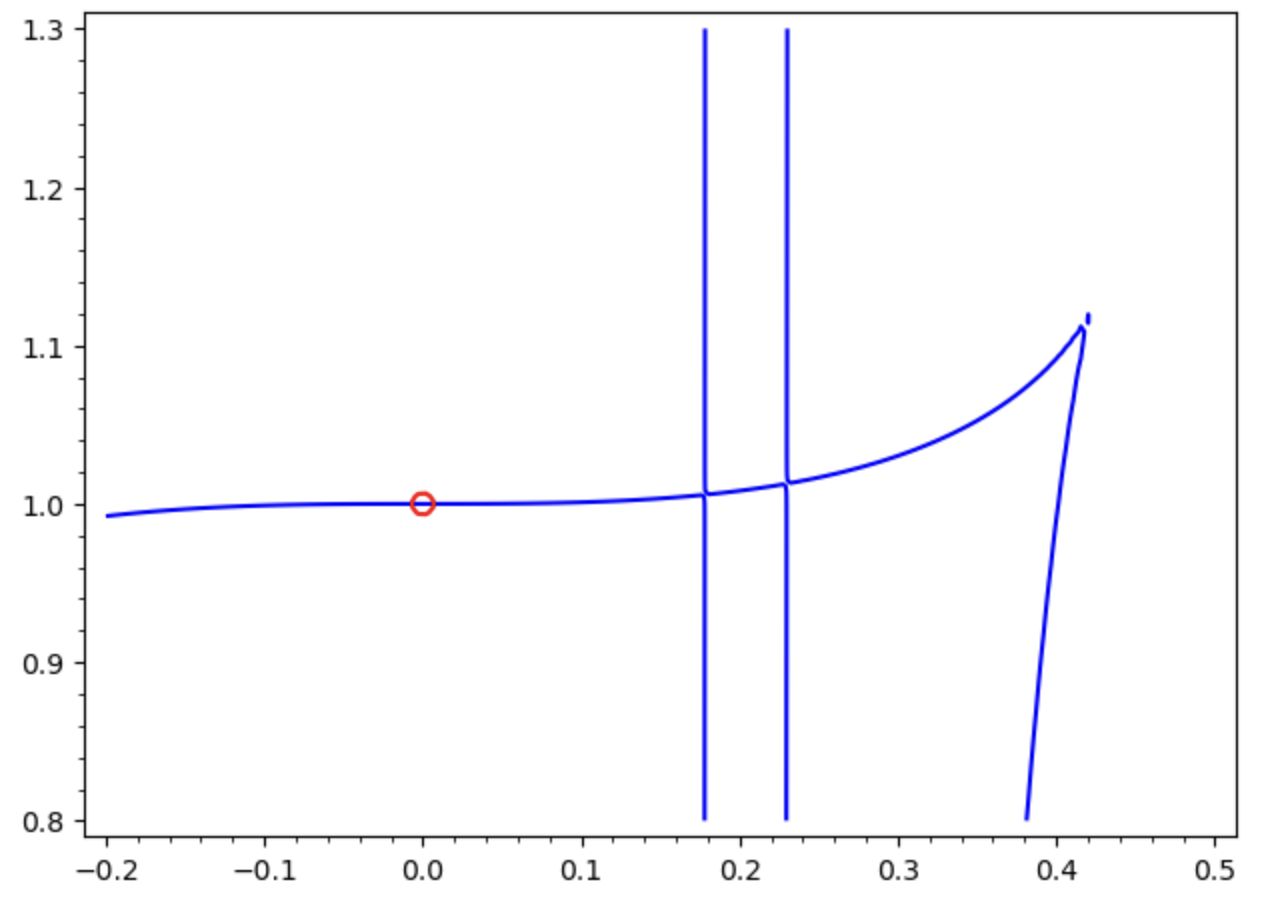}
    \includegraphics[width=0.2\linewidth]{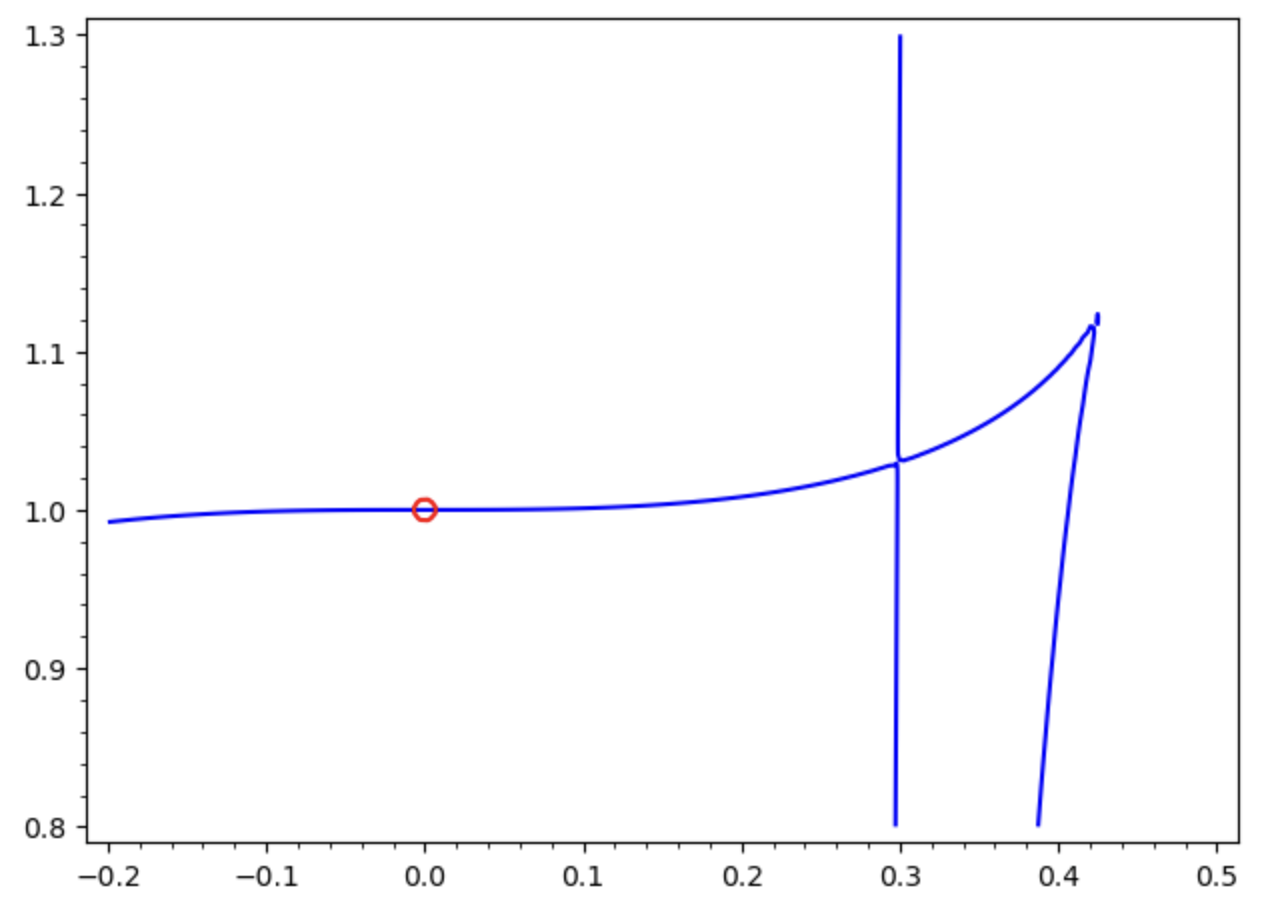}
    \caption{The solution $T(z,0,x)$ for the values $x= 0, 0.025, 0.05, 0.1$ and in the second line $x=-0.005,-0.01,-0.025, -0.05$. The red dot marks the origin.}
    \label{fig:plots_t0}
\end{figure}

We make now a rigorous analysis of these solution functions.
Actually, if $x\to 0$ then it is clear from the equation 
$P(z,u,T,t_0,t_1,x) = 0$ 
that one of the solutions $u_1(z,x)$, $u_2(z,x)$ tends to $0$ as $x\to 0$; we denote
this solution as $u_2(z,x)$. Since 
\[
P_T(z,u_2,\tilde T_2,t_0,t_1,x) = 2zu_2^3 \tilde T_2 + zu_2 - u_2^2 + zx = 0
\] 
it follows that 
\[
u_2(z,x) \approx -x
\]
if $z \ne 0$. Actually it is easy to see that $u_1(z,x) \sim \sqrt{zx}$
and $u_2(z,x) \sim -\sqrt{zx}$ if $|z|\le |x|$. If $|z|\ge |x|$ then 
the behavior changes to $u_1(z,x) \sim u_1(x)$ and $u_2(z,x) \sim -x$ 
(as $x\to 0$).

Assume for a moment that $x\ne 0$. Then we use the second set of equations
\begin{equation}\label{eqsecondset}
    P(z,u_2,\tilde T_2,t_0,t_1,x) = 0, \quad P_u(z,u_2,\tilde T_2,t_0,t_1,x) = 0, \quad 
P_T(z,u_2,\tilde T_2,t_0,t_1,x) = 0
\end{equation}
and solve them with respect to $u_2$, $\tilde T_2$, and $t_1$, that is, we consider
$z,x,t_0$ as parameters and obtain solution functions
\[
\overline u_2(z,x,t_0), \quad \overline T_2(z,x,t_0), \quad \overline t_1(z,x,t_0).
\]
Actually we only need the function $\overline t_1(z,x,t_0)$ that we
substitute into the first three equations:
\begin{align*}
P(z,u_1, \tilde T_1, t_0, \overline t_1(z,x,t_0), x) &= 0, \\
P_u(z,u_1, \tilde T_1, t_0, \overline t_1(z,x,t_0), x) &= 0, \\
P_T(z,u_1, \tilde T_1, t_0, \overline t_1(z,x,t_0), x) &= 0.
\end{align*}
If we set $R(z,u,T,t_0,x) = P(z,u,T,t_0,\overline t_1(z,x,t_0),x)$ then this
just reads as 
\begin{equation}\label{eqRsystem}
R(z,u,T,t_0,x) = 0, \quad R_u(z,u,T,t_0,x) = 0, \quad R_T(z,u,T,t_0,x) = 0
\end{equation}
which is precisely a system of the kind that appears in \cite{DNY2021} in order
to handle first order DDE's. As mentioned above the essential point is to 
show that this system behaves {\it smoothly} as $x\to 0$ and is equivalent to
the system (\ref{eqP1system}) for $x=0$. 

Let's start with the elimination process that computes 
$\overline u_2(z,x,t_0)$, $\overline T_2(z,x,t_0)$, and $\overline t_1(z,x,t_0)$.
For this purpose we replace the functions $\overline u_2$ and  $\overline T_2(z,x,t_0)$ with
the function $V(z,x,t_0)$ and $N(z,x,t_0)$ by setting
\[
\overline u_2(z,x,t_0) = x V(z,x,t_0),\quad 
\overline T_2(z,x,t_0) = t_0 + x N(z,x,t_0).
\]
With this choice the system (\ref{eqsecondset}) rewrites to 
\begin{align}
    V^2 + z^2x V^3 (t_0+xN)^2 + zVN - V^2(t_0+xN) + z(N - Vt_1) &=0, \nonumber \\
    2V + 3z^2xV^2 (t_0+xN)^2 + zN - 2V(t_0+xN) - zt_1 &= 0, \label{eqsecondset-2} \\
    2z^2x^2V^3(t_0+xN) + zV - xV^2 + z &= 0.  \nonumber
\end{align}
For $x=0$ (and arbitrary $z\ne 0$) this system has the solution
\[
V_0 = -1, \quad N_0 = \frac{1-t_0}{z}, \quad t_1 = \frac{t_0-1}{z}.
\]
The Jacobian of the system (\ref{eqsecondset-2}) is given by the matrix 
\[
J = \left(    
\begin{array}{ccc}
0    &    0    &   -zV  \\
2 + 6z^2xV (t_0+xN)^2 - 2(t_0+xN)   &  6z^2x^2V^2 (t_0+xN) + N - 2xV & 2V  - z  \\
6z^2x^2V^2(t_0+xN) + z - 2xV & 2z^2x^3V^3N  & 0
\end{array}
\right)
\]
and its determinant by
\begin{align*}
      \det J &= -zV \big(  (2 + 6z^2xV (t_0+xN)^2 - 2(t_0+xN)) 2z^2x^3V^3N \\
     &   \qquad \qquad  
- (6z^2x^2V^2 (t_0+xN) + N - 2xV )^2
\big).
\end{align*}
For $x=0$ the determinant has value $zV_0 N_0^2 = - (t_0-1)^2/z$. So it is non-zero 
for $z\ne 0$ and $t_0 \ne 1$. Actually we are interested in the case $z> 0$ and
$t_0 > 1$ (note that $t_0(0,0) = 1$ and that $t(z,0)$ is strictly increasing). 
Thus, we can assume that the determinant is non-zero for $z> 0$, $t_0> 1$ and for
$x$ sufficiently close to $0$. Hence, by the implicit function theorem 
there are analytic solutions $V(x,z,t_0)$, $N(x,z,t_0)$, $\overline t_1(x,z,t_0)$ 
with $V(0,z,t_0) = -1$, $N(0,z,t_0) = (1-t_0)/z$, $\overline t_1(0,z,t_0) = (t_0-1)/z$.
As explained above we use now the function $\overline t_1(x,z,t_0)$, insert it into
the into the first three equations and obtain the system (\ref{eqRsystem}).
For $x=0$ this is precisely the system (\ref{eqP1system}), where we already
know that $t_0(z,0)$ as a proper $3/2$-singularity. Now we can directly apply
Theorem 2 from \cite{DHai} and obtain that $t_0(z,x)$ has a representation of the form  (\ref{eq32expansion}) which is then sufficient to obtain a central limit theorem
provided that $z_0'(1) \ne 0$. However, by using the system (\ref{eqRsystem}) together
with the equation $R_{uu}R_{TT}-R_{uT}^2 = 0$ and by implicit differentiation 
this can be easily checked. 

\medskip

In general we can proceed in a similar way. We start with the case $k=2$ 
(again with $x$ instead of $x-1$ and with 
$T(z,u,x)$ instead of $F(z,u,x)$):
\begin{align}
        T(z,u,x) &= z Q\left(z,u, T(z,u,x), \Delta T(z,u,x) \right)    \label{eqpertDDE-2} \\
        &\qquad +xzR\big(z,u,x,T(z,u,x),\Delta T(z,u,x), \Delta^2 T(z,u,x)\big). \nonumber
\end{align}
Equivalently we have the equation
\[
P(z,u,T,t_0,t_1,x) = zQ\left( z,u,T, \frac{T-t_0}u \right) + 
xz R\left( z,u,x, T,\frac{T-t_0}u, \frac{T-t_0-ut_1}{u^2}  \right) - T = 0.
\]
Again if $x=0$ this equation simplifies to 
\[
P_1(z,u,T,t_0) = zQ\left( z,u,T, \frac{T-t_0}u \right) 
 - T = 0.
\]
Note that have not multiplied by $u$ or $u^2$ in order to write it 
as a polynomial equation. Actually we use this multiplication in a later step. 
By using the methods of \cite{DrmotaNoyYu} it follows that
$t_0(z,0) = T(z,0,0)$ has a dominant $3/2$-singularity at some $z_0> 0$ (that can be
computed by solving the system $P_1 = P_{1,u} = P_{1,T} = 
P_{1,uu}P_{1,TT} - P_{1,uT}^2 = 0$). 

For $x\ne 0$ the order of the DDE equals $2$ so that we have to solve the system
\begin{align*}
P(z,u_1,\tilde T_1,t_0,t_1,x) &= 0, \quad P_u(z,u_1,\tilde T_1,t_0,t_1,x) &= 0, \quad 
P_T(z,u_1,\tilde T_1,t_0,t_1,x) &= 0, \\
P(z,u_2,\tilde T_2,t_0,t_1,x) &= 0, \quad P_u(z,u_2,\tilde T_2,t_0,t_1,x) &= 0, \quad 
P_T(z,u_2,\tilde T_2,t_0,t_1,x) &= 0
\end{align*}
for the unknown functions
$u_1(z,x)$, $u_2(z,x)$, $t_0(z,x)$, $t_1(z,x)$, $\tilde T_1(z,x) = T(z,u_1(z,x),x)$, and
$\tilde T_2(z,x) = T(z,u_2(z,x),x)$. From the equation
\[
P_T = zQ_{y_0} + \frac{zQ_{y_1}}u + x \left( zR_{y_0} + \frac{zR_{y_1}}u 
+ \frac{zR_{y_2}}{u^2} \right) - 1 = 0
\]
that we multiply by convenience by $u^2$
\[
u^2 P_T = zu^2Q_{y_0} + zuQ_{y_1} + x \left( zu^2R_{y_0} + {zuR_{y_1}} 
+ {zR_{y_2}} \right) - u^2 = 0
\]
it follows (similarly as in the above example) that the solutions
$u_2(z,x)$ satisfies (for $z\ne 0$)
\[
u_2(z,x) \approx - x \frac{R_{y_2}}{Q_{y_1}}.  
\]
Recall that by assumptions $Q_{y_1} \ne 0$ and $R_{y_2} \ne 0$.
Thus, we are in precisely the same situation as above. We set 
$\overline u_2 = xV$ and $\overline T_2 = t_0 + xN$. By using the assuption that $R$ 
is linear in $y_2$ it follows that the system of the second third equations
can be rewritten into a system for $V$, $N$, and $t_2$ with the property that
there is a smooth transition to $x=0$. For example, the first equation
rewrites to 
\[
z Q(z,xV,t_0 + xN, N/V) + zx z Q(z,xV,s,t_0 + xN, N/V,(N-Vt_1)/(xV^2)) - t_0- xV = 0,
\]
where due to the linearity in $y_2$ the factor $1/x$ cancels.
As above this (new) system can be solved for $x=0$ and
we get analytic solutions $V = \overline V(z,x,t_0)$, $N = \overline N(z,x,t_0)$, 
$t_2 = \overline t_2(z,x,t_0)$ etc.

This procedure works in a very similar way if $k> 2$ but one has to take care
of an additional feature. In this case we have to deal with a system of $3k$ equations
with solution function $u_1,\ldots,u_k$, $\tilde T_1, \ldots, \tilde T_k$, $t_1, \ldots, t_k$,
where the solutions functions $u_2,\ldots,u_k$ are now of order $x^{1/(k-1)}$. 
Again one has solve first the last $3(k-1)$ equations with parameters $z,x,t_0$.
Here we use the substitutions $\overline u_j = X V_j$ and 
$\overline T_j = t_0 + X V_j t_1 + \cdots + X^{k-1} V_j^{k-1} t_{k-1} + 
X^k N_j$, $2\le j\le k$,
where $X = x^{1/(k-1)}$. As in the previous cases we get analytic solution functions
for $z\ne 0$ and for $X$ close to zero. Furthermore, the solution functions
$\overline t_j$, $2\le j\le k$, depend only on $X^{k-1} = x$, that is, we can 
finally reduce the problem to a system of $3$ equations of the type (\ref{eqRsystem}), and we are done.

%----------------------------------------------------------

\section{Pattern counts in simple triangulations}\label{sec:trian}

In this section, we apply the theorem in Section 2 to the counts of pattern occurrences in simple triangulations. It is known that pattern counts of triangulations, where the patten cannot self-intersect, satisfy a central limit theorem~\cite{GaoWormald}. In the following, we also consider maps with arbitrary root face degree as patterns in a map and establish first a functional equation for the generating function of simple triangulations with an additional variable counting pattern occurrences of a map which cannot self-intersect. Subsequently, we derive a central limit theorem for patterns with root face valency $4$ from our Theorem~\ref{thm:clt1dde}.

\begin{definition}
    A \emph{planar map} is a connected planar graph (with loops and multiple edges allowed) embedded onto the sphere. If one of its edges is oriented, we call it \emph{rooted}. The oriented edge is further called the \emph{root edge} and the vertex from which the root edge is pointing away the \emph{root vertex}.
    A planar map separates the surface into several connected regions called \emph{faces}. The face to the left of the root edge is called the \emph{root face} or the \emph{exterior face}.
    The \emph{valency} of a face is the number of edges incident to it, bridges being counted twice. A face of valency $m$ is called an \emph{$m$-gon}.
    We define the \emph{boundary} of a rooted map as the set of all edges and vertices incident with the root face.
    A \emph{simple map} is a map without multiple edges or loops. A \emph{simple near-triangulation} is a simple map in which all faces except the root face have degree 3 and there are no interior edges connecting two vertices of the boundary. If the root face of a simple near-triangulation is of degree 3, we call the map a \emph{simple triangulation}. We denote the set of simple near-triangulations by $\mathcal{T}$ and their generating function by
        \[
            T(z,u) = \sum_{n,j\geq 0} t_{j,n}u^jz^n,
        \]
    where $u$ marks the root face valency minus $3$, $z$ marks the number of interior edges minus the root face valency plus $3$ (the exponent of $u$), divided by $3$. 
    
    These counting variables were carefully chosen by Tutte in~\cite{Tutte1962a} with respect to the decomposition of these maps and the singularity analysis of the resulting functional equation.
\end{definition}

Ultimately, we are interested in the number of pattern occurrences of a fixed map ${\bf p}$ in a random planar triangulation.

\begin{definition}[Pattern occurrences~\cite{BenderGaoRichmond} ]\label{def:submaps}
    Let~\map{p} be a rooted map. We say that~\map{p} occurs as a pattern in a
    map~\map{m} if~\map{m} can be obtained by extending~\map{p} in the following way:
    \begin{enumerate}
        \item[(a)] adding vertices to the interior of the root face of~\map{p},
        \item[(b)] adding edges with their endpoints being either vertices or edges from the
        boundary of~\map{p} or newly created vertices,
        \item[(c)] rerooting the so obtained map in such a way that its new root face is not contained in an interior face of~\map{p}.
    \end{enumerate}
    We say, $\map{p}$ cannot \emph{self-intersect} if any interior face of a pattern occurrence of $\map{p}$ cannot be the interior face of another pattern occurrence unless they have a root face preserving isomorphism (see Figure~\ref{fig:placeholder} for an example).
\end{definition}

\begin{proposition} \label{thm:patts}
    Let $\map{p}$ be a near-triangulation with $e$ edges, root face valency $v$ and $r$ rotational symmetries which cannot self-intersect. Then the generating function $T(z,u,x)$ for planar near-triangulations enriched with the variable $x$ that counts the number of occurrences of $\map{p}$ as a pattern satisfies the equation
    \begin{align*}
        T(z,u,x) &= 1+z\Delta T(z,u,x) + z\frac{T(z,u,x)^2}{1-uT(z,u,x)}\\
        &\qquad+r(x-1)\frac{z^{\frac{e+v-2}{3}}}{u^{v-2}}N_{v-1}(z,u,T(z,u,x),\Delta T(z,u,x), \dots, \Delta^{v-1}T(z,u,x))
    \end{align*}
    where $\frac{z^{\frac{e+v-2}{3}}}{u^{v-2}}N_{v-1}(z,u,y_0,\dots,y_{v})$ is a polynomial in $z,y_1,\dots,y_v$ and rational in $u$ and $T$.
\end{proposition}

\begin{proof}
    The first terms of the equation count triangulations where the root edge is not incident to any pattern occurrences. These cases were treated in~\cite{Tutte1962a} and consist of one of the following three cases:
    \begin{enumerate}
        \item a single triangle marked by $1$ since the root face valency is $3$ and there are no interior edges (and thus $u^{3-3}z^{(0-0)/3}=1$).
        \item a (near-)triangulation with root face valency at least $4$ without interior edges connecting two vertices on its boundary, to which we attach a new root edge, decreasing the root face valency by $1$ and increasing the number of interior edges by $2$. Thus, this case is counted by $z\Delta T(z,u,x)$.
        \item a near-triangulation with interior edges connecting the vertex to which the root edge points to with at least one other point on the boundary to which we attach a new root edge. This near-triangulation can in turn be decomposed into a sequence of at least two triangulations without interior edges connecting two vertices on the boundary. Careful computations on the number of interior edges and root face valency yield the term $z\frac{T(z,u,x)^2}{1-uT(z,u,x)}$.
    \end{enumerate} 
    The last term corresponds to the case where the root edge is incident to a pattern occurrence. If we delete the interior edges of this pattern occurrence, we are left with a map that decomposes into a sequence of single edges or maps that themselves decompose into a non-empty sequence of near-triangulations glued along edges (see Figure~\ref{fig:triang}), which we call \emph{multifan-triangulations}. 
    \begin{figure}
        \centering
        \includegraphics[width=\linewidth, page=1]{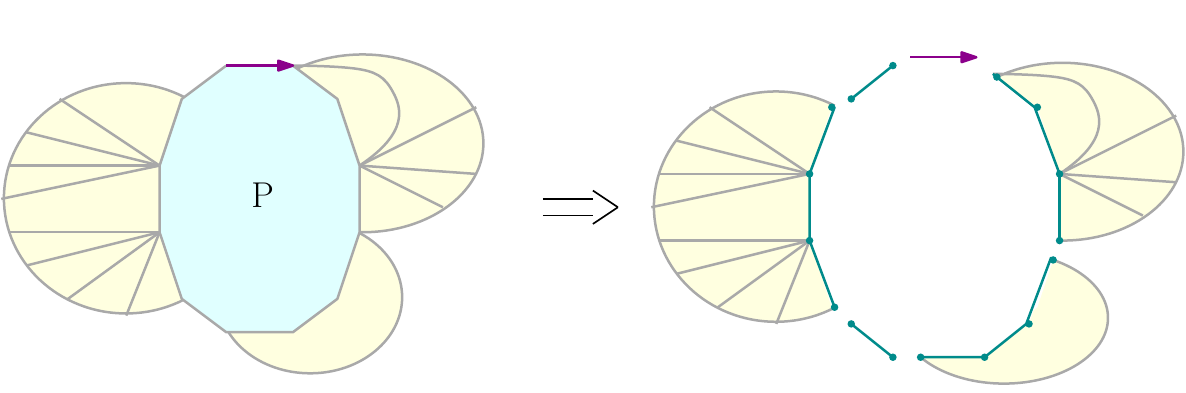}
        \caption{Decomposition of triangulations with a root edge incident to a labeled pattern occurrence into a sequence of edges and multifan-triangulations}
        \label{fig:triang}
    \end{figure}
    We denote their generating function by $F(z,u,x,w)$, where $z,u,x$ as above and $w$ is an auxiliary variable that marks the length of the boundary which is incident to the deleted pattern occurrence. 
    
    We now aim to describe $F(z,u,x,w)$ in terms of $T(z,u,x)$ and its discrete derivatives. \\
    First of all, if $\ell$ is the root face valency of a multifan-triangulation, then the number of edges on the boundary incident to the pattern occurrence is at least $2$ and at most $\ell-1$ (otherwise an interior edge would connect two vertices on the boundary of the original equation).
    Further, any vertex on the boundary where both incident boundary edges lie on the boundary of the pattern occurrence, might be connected to another vertex on the boundary of the multifan-triangulation. Interior edges connecting vertices on the boundary of the multifan-triangulation therefore decompose the map into sequences of near-triangulations glued together along the edges. Since the image that we can draw looks like a sequence of fans glued together, we chose the name multifan-triangulations. 
    In particular, we decompose the multifan-triangulation into a near-triangulation with at least one edge incident to the pattern occurrence and an alternating sequence of sequences of near-triangulations with no edges incident to the pattern occurrence and near-triangulations with at least one edge incident to a pattern occurrence, ending with the latter. (See Figure~\ref{fig:multifan})

    \begin{figure}
        \centering
        \includegraphics[width=0.75\linewidth, page=2]{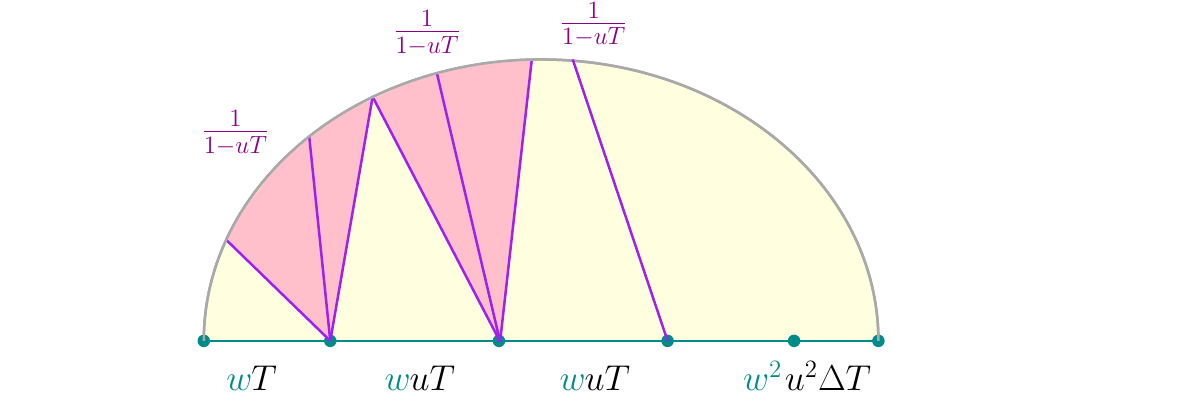}
        \caption{Decomposition of multifan-triangulations into an alternating sequence of near-triangulations incident to the pattern occurrence and sequences of near-triangulations}
        \label{fig:multifan}
    \end{figure}
    
    Note that none of these gluing operations create/destroy pattern occurrences since the pattern is a near-triangulation without interior edges connecting vertices on its boundary.
    
    We therefore only have to compute the exponents of $z$ and $u$ resulting from gluing near-triangulations.

    So, assume we glue two maps with root face valencies $v_1$ and $v_2$ and interior edges $i_1, i_2$ respectively, along an edge incident to a vertex $v$. Then the root face valency of the obtained map equals $v_1+v_2-2$. So, $u$ counting the root face valency - $3$ gets the exponent $v_1+v_2-5 = (v_1-3)+(v_2-3)+1$ and therefore we gain a factor $u$. The number of interior edges on the other hand increases by $1$ and therefore the exponent of $z$, counting the number of interior minus the exponent of $u$, divided by three sums up perfectly as
    \[
        \frac{i_1+i_2+1-v_1-v_2+5}{3}=\frac{(i_1-v_1+3)+(i_2-v_1+3)}{3}.
    \]
    Therefore, no additional factors of $z$ will appear. 
    
    By the same arguments, we arrive at the conclusion that gluing a sequence of maps results in an additional factor $u$ per map. 

    Finally, we note that if $\ell$ edges of a near-triangulation in the multifan-triangulation are incident to the pattern occurrence, the root face valency of the near triangulation has to be at least $\ell-1$ and no multifan-triangulation has only a single edge incident to the pattern occurrence.
    
    Hence, the generating function is described by
    \begin{align*}
        F(z,u,x,w) 
        &= \frac{wT(z,u,x)+\sum_{k\geq 0} w^{k+2}u^{k}\Delta^{(k)}T(z,u,x)}{1-\frac{u}{1-uT(z,u,x)}\left(wT(z,u,x)+\sum_{k\geq 0} w^{k+2}u^k\Delta^{(k)}T(z,u,x)\right)} - wT(z,u,x)
    \end{align*}
    The map which is left over is then decomposed into a sequence of edges and multifan-triangulations excluding a single multifan-triangulation with root face valency $\ell>v$ and $\ell-1$ edges incident to the root face. When we set up the generating function for this sequence, we multiply each generating function of a multifan-triangulation by $u^3$, and after combining divide by $u^3$ to correctly produce the exponent of $u$ corresponding to the root face valency minus $3$. Further, when we consider the exponent of $z$ each multiplication of $u^3$ corresponds to an increase of $1$ in the exponent of $z$.
    
    An edge corresponds to a term $z^{-\frac{2}{3}}u^2$ as it increases the root face valency by $2$. When we glue the pattern to the left over map, the edges do not produce new interior edges, but the edges on the boundaries of the fan triangulations contribute each $z^{\frac{1}{3}}$ to the generating function. In total, we thus get
    \begin{align*}
        N_{v-1}(z,u,x) &= \frac{z}{u^3}[w^{v-1}u^{\geq(v+1)}] \frac{z^{-\frac{2}{3}}u^2w+z^{-1}u^3F(z,u,x,z^{\frac{1}{3}}w)}{1-(z^{-\frac{2}{3}}u^2w+z^{-1}u^3F(z,u,x,z^{\frac{1}{3}}w))}
    \end{align*}
    which in turn can be expressed as a polynomial in $T(z,u,x),\Delta T(z,u,x), \dots \Delta^{(k)} T(z,u,x)$, since any function $[u^\ell] T(z,u,x) = \Delta^{(\ell)} T(z,u,x) - u\Delta^{(\ell+1)} T(z,u,x)$.
    
    Now, let the left-over map consist of $v-m$ single edges and the rest fan triangulations. In that case, the exponent of $z$ in the above expression is $z^{(m-v)\frac{2}{3}+i}$, where $i\in \mathbb{Z}_{\geq -1}$. By gluing the pattern to the left over map, we decrease the exponent of $u$ by $v-2$ and consequently increase the one of $z$ by $\frac{v-2}{3}$ and further, we create $m-1$ new internal edges, such that the new exponent of $z$ will be
    \[
        z^{\frac{e}{3}+\frac{v-2}{3}+(m-v)\frac{2}{3}+i+\frac{m-1}{3}} = z^{\frac{e-v+3m+3i-3}{3}}.
    \]
    Since \map{p} is a near-triangulation $e-v$ is divisible by three and the above is indeed an integer.
\end{proof}

We illustrate once again the decomposition procedure along an example.
\begin{figure}
        \centering
    \includegraphics[width=0.2\linewidth]{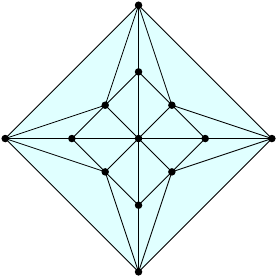}
    \caption{The diamond pattern which cannot self-intersect}
    \label{fig:placeholder}
\end{figure}

\begin{example}[Diamonds] A diamond is a rotational symmetric near-triangulation with root face valency four and $28$ internal edges as depicted in Figure~\ref{fig:placeholder}. It is easy to see that it cannot self-intersect, since none of the interior vertices can lie on the boundary of a second pattern occurrence of a diamond.

By Proposition~\ref{thm:patts}, the generating function of near-triangulations with $x$ marking the number of diamonds satisfies the equation
    \begin{align}\label{eq:diam}
        T(z,u,x) &= 1+z\Delta T(z,u,x) + z\frac{T(z,u,x)^2}{1-uT(z,u,x)} \nonumber\\
        &\qquad+\frac{(x-1)z^{\frac{28+2}{3}}}{u^2}F_{3}(z,u,T(z,u,x),\Delta T(z,u,x), \Delta^{2}T(z,u,x)),
    \end{align}
    where
    \begin{align*}
        &F_{3}(z,u,T(z,u,x),\Delta T(z,u,x), \Delta^{2}T(z,u,x)) \\
        &=\frac{z}{u^3}\left(z^{-2}u^6+2z^{-\frac{2}{3}}u^2\frac{u^3}{z}\left(z^{\frac{2}{3}}T+\frac{z^{\frac{2}{3}}uT^2}{1-uT}\right)+2\frac{u^3}{z}\frac{zu^2T\Delta T}{1-uT}+\frac{u^3}{z}\frac{zu^2T^3}{(1-uT)^2}+\frac{u^3}{z}zu^2\Delta^{(2)} T\right)\\
        &=\frac{u^3}{z}+2u^2\left(T+\frac{uT^2}{1-uT}\right)+2z\frac{u^2T\Delta T}{1-uT}+z\frac{u^2T^3}{(1-uT)^2}+zu^2\Delta^{(2)}T
    \end{align*}
    Indeed, $F_3$ corresponds to the following cases 
    \begin{enumerate}
        \item the map itself is a diamond
        \item on of the edges left or right to the root edge is a single edge attached to a 
        \begin{enumerate}
            \item near-triangulation
            \item multifan triangulation with a single fan
        \end{enumerate}
        \item all edges in the left-over map are incident to a multifan triangulation with 
        \begin{enumerate}
            \item a single fan at one of the two interior vertices along the glued boundary
            \item fans at each interior vertex along the boundary
        \end{enumerate}
        \item a near triangulation with root face valency at least $5$.
    \end{enumerate}
    By Theorem~\ref{thm:clt1dde}, we can conclude that the number of diamond occurrences has an expectation and variance linear in $n$ and satisfies a central limit theorem.\\
    
    Note that the fact that the exponent of $z$ in~\eqref{eq:diam} is divisible by $3$ is due to the fact that the root face valency $4$ is $1$ modulo $3$. For a non-intersecting pattern with root face valency $5$ the exponent of $z$ in~\eqref{eq:diam} would indeed be of the form $\frac{3m+2}{3}$ for some non-negative $m\in \mathbb{Z}$ and only add up to an integer in combination with the exponents of $z$ in $F_4$. This is easy to see if one considers the case, where the map is in fact the pattern and the corresponding term in $F_4$ contributes $zu^{-3}z^{-\frac{8}{3}}u^8 = z^{-\frac{5}{3}}u^5$.
\end{example}

Finally, we directly obtain the following central limit theorem for pattern counts of non-selfintersecting patterns in simple triangulations.

\begin{theorem}
    Let $\map{p}$ be a near-triangulation which cannot self-intersect and $X_n$ the (random) number of occurrences of $\map{p}$ in a uniformly random simple triangulation with $n$ edges. Then there exist computable constants $\mu, \sigma>0$ such that
    \[
        \frac{X_n - \mu n}{\sqrt{n}} \longrightarrow {N}(0,\sigma^2)
    \]
    where ${N}(0,\sigma^2)$ is a normal random variable with mean $0$ and variance $\sigma^2$. 
\end{theorem}

\section{Conclusions}
We established that singularly perturbed discrete differential equations admit a smooth transition from order $1$ to order $k\ge2$ in the case where the equation is linear in the highest discrete derivative. Under the usual assumptions of strong connectedness and non-degeneracy, the dominant singularity varies analytically with the perturbation parameter $x$, which allows us to transfer the singular expansion and derive a central limit theorem for associated counting parameters.
A natural open problem is whether such a smooth transition persists without the linearity assumption in the highest-order discrete derivative. In that setting, the elimination procedure becomes substantially more involved and requires careful handling.
On the combinatorial side, we restricted attention to patterns in triangulations that cannot self-intersect. Allowing self-intersecting patterns leads to more intricate decompositions, where the resulting functional equations become technically cumbersome due to the delicate choice of counting variables and the restriction to simple triangulations. In particular, the highest discrete derivative does not necessarily appear linearly in the equation. An alternative approach based on probabilistic arguments appears promising and is currently under investigation in joint work with Nicolas Curien and Alice Contat.

Further applications of Theorem~\ref{thm:clt1dde} and extensions of it are expected in various other combinatorial structures with generating functions satisfying discrete differential equations such as fighting fish or Tamari intervals.

%++++++++++++++++++++++++++++++++++++++++++++++++++++++++++
%++++++++++++++++++++++++++++++++++++++++++++++++++++++++++
\section{References}
\bibliography{Bib_maps.bib}{}
\bibliographystyle{plain}
%++++++++++++++++++++++++++++++++++++++++++++++++++++++++++
%++++++++++++++++++++++++++++++++++++++++++++++++++++++++++
\end{document}